\documentclass[a4paper,oneside,11pt]{article}

\usepackage{amsmath,amsfonts,amscd,amssymb}
\usepackage{longtable,geometry}
\usepackage[english]{babel}
\usepackage[utf8]{inputenc}
\usepackage[active]{srcltx}
\usepackage[T1]{fontenc}
\usepackage{graphicx}
\usepackage{pstricks}
\usepackage{bbm}
\usepackage{MnSymbol}

\newtheorem{theorem}{Theorem}[section]
\newtheorem{conj}{Conjecture}

\newtheorem{corollary}[theorem]{Corollary}
\newtheorem{lemma}[theorem]{Lemma}
\newtheorem{proposition}[theorem]{Proposition}
\newtheorem{definition}[theorem]{Definition}

\newtheorem{remark}[theorem]{Remark}

\newtheorem{question}[conj]{Question}
  
\newcommand{\be}[1]{\begin{equation}\label{#1}}
\newcommand{\ee}{\end{equation}}
\numberwithin{equation}{section}

\newenvironment{proof}[1][\relax]%s
  {\paragraph{Proof\ifx#1\relax\else~of #1\fi}}%
  {~\hfill$\square$\par\bigskip}
  {\paragraph{Proof\ifx#1\relax\else~de #1\fi}}%
  {~\hfill$\square$\par\bigskip}

\newcommand{\bbC}{\mathbb{C}}

\newcommand{\bbR}{\mathbb{R}}

\newcommand{\bbZ}{\mathbb{Z}}
\renewcommand{\be}{\beta}

\renewcommand{\Im}{\textrm{Im}}
\renewcommand{\Re}{\textrm{Re}}

\newcommand{\rk}[1]{\bgroup\color{red}%
  \par\medskip\hrule\smallskip%
  \noindent\textbf{#1}%
  \par\smallskip\hrule\medskip\egroup}

\title{A functional relation for $L$-functions of graphs equivalent to the Riemann Hypothesis for Dirichlet $L$-functions}
\author{Fabien Friedli%
	\thanks{The author was supported in part by the Swiss NSF grant 200021 132528/1.%
	}}

\begin{document}
\maketitle

\begin{abstract}
In this note we define $L$-functions of finite graphs and study the particular case of finite cycles in the spirit of a previous paper that studied spectral zeta functions of graphs. The main result is a suggestive equivalence between an asymptotic functional equation for these $L$-functions and the corresponding case of the Generalized Riemann Hypothesis. We also establish a relation between the positivity of such functions and the existence of real zeros in the critical strip of the classical Dirichlet $L$-functions with the same character.
\end{abstract}

\begin{section}{Introduction}

In a previous paper with Karlsson (\cite{FK16}), we initiated the study of spectral zeta functions of graphs. In particular, we studied the behaviour of the zeta function of the cyclic graph with $n$ vertices as the parameter $n$ goes to infinity. It turned out that the classical Riemann zeta function (or Epstein zeta functions in more than one dimension) appeared in the asymptotics, together with the spectral zeta function of the graph $\bbZ$. This lead to a surprising reformulation of the Riemann Hypothesis in terms of an asymptotic functional equation $s\longleftrightarrow 1-s$ for spectral zeta functions. It is natural to wonder whether this could be generalized to $L$-functions. 
\\

\noindent
In this work, we propose a definition of $L$-functions for finite graphs different from those in Stark-Terras \cite{ST00}. Given a finite graph $G$ with $m\geq 3$ vertices and a primitive character modulo $k\geq 3$ one can define a \emph{(spectral) $L$-function} of $G$ by $$L_G(s,\chi):=\sum_{j=1}^{m-1}\frac{\chi(j)}{\lambda_j^s},$$ where the sum runs over all non-zero ordered eigenvalues of the combinatorial Laplacian on $G$ and $s$ is any complex number, in analogy with the definition of spectral zeta functions in \cite{FK16}. For our purposes here we normalize things a bit differently in the case of the Cayley graph of $\bbZ/kn\bbZ$ and a character modulo $k$.

\begin{definition}
	\emph{Let $k\geq 3$ and let $\chi$ be a primitive and even ($\chi(-1)=1$) Dirichlet character modulo $k$. For $n\geq 1$, the function $L_n:\bbC\longrightarrow\bbC$ defined by $$L_n(s,\chi):=\sum_{j=1}^{kn-1}\frac{\chi(j)}{\sin(\frac{\pi j}{kn})^{s}}$$ is the \emph{(normalized) $L$-function} of the cyclic graph with $kn$ vertices $\bbZ/kn\bbZ$ associated to the character $\chi$.}
\end{definition}
In contrast to \cite{FK16}, it is not clear to us how to define appropriate $L$-functions for infinite graphs.\\
Denote by $L(s,\chi)$ the classical Dirichlet $L$-function associated to the character $\chi$, which is defined by $$L(s,\chi)=\sum_{j\geq 1}\frac{\chi(j)}{j^s}$$ when $\Re(s)>0$ and by analytic continuation elsewhere. We will prove in section 2 the following asymptotics:
\begin{proposition}\label{MainProp}
	For any $s\in\bbC$ we have, as $n\longrightarrow\infty$, $$L_n(s,\chi)=2\left(\frac{kn}{\pi}\right)^{s}\left(L(s,\chi)+\frac{s}{6}\left(\frac{kn}{\pi}\right)^{-2}L(s-2,\chi)+O(\frac{1}{n^4})\right).$$
\end{proposition}
This leads us into an interesting reformulation of the Generalized Riemann Hypothesis (GRH) (reminiscent to Riemann's case in \cite{FK16}, although there one also has a $\zeta_{\bbZ}$ term).
\begin{theorem}\label{MainThm}
	Let $\chi$ be a primitive and even Dirichlet character modulo $k\geq 3$. For $0<\Re(s)<1,$ let $$\xi_n(s,\chi)=n^{-s}(\frac{\pi}{k})^{\frac{s}{2}}\Gamma(\frac{s}{2})L_n(s,\chi)$$ be the (normalized) completed $L$-function of $\bbZ/kn\bbZ$. The following are equivalent:
	\begin{description}
		\item[(i)] For all $s$ such that $0<\Re(s)<1$ and $\Im(s)\geq 8$ we have $$\lim\limits_{n \to \infty}\frac{|\xi_n(s,\chi)|}{|\xi_n(1-s,\overline{\chi})|}=1;$$ 
		\item[(ii)] In the region $\Im(s)\geq 8$, $L(s,\chi)$ satisfies the Generalized Riemann Hypothesis (GRH), that is all zeros in this region have real part $\frac{1}{2}$. 
	\end{description}
\end{theorem}
We will see in the proof that the relation in (i) holds in any case for all $s$ such that $L(s,\chi)\neq 0$ and for obvious reasons on the critical line $\Re(s)=\frac{1}{2}$ for each $n$.\\

\noindent
For comparison, recall that if we let $\xi(s,\chi)=(\frac{\pi}{k})^{\frac{-s}{2}}\Gamma(\frac{s}{2})L(s,\chi)$ then the functional equation implies that $|\xi(s,\chi)|=|\xi(1-s,\overline{\chi})|$.\\

\noindent
In view of Proposition \ref{MainProp} there is a more obvious way to relate properties of our $L$-functions of cyclic graphs to GRH. Indeed, since the holomorphic functions $\frac{1}{2}\left(\frac{\pi}{kn}\right)^sL_n(s,\chi)$ converge to the corresponding Dirichlet $L$-function, it would be possible to study zeros of these finite sums to get information about zeros of $L$ (for example using Hurwitz theorem). Also we see that the rate of convergence of the sequence $L_n$ changes according to whether $L(s,\chi)=0$ or not. Nervertheless we believe the result given in Theorem \ref{MainThm} is more unexpected and structural, in that it unveils a relationship between zeros of $L$ and an asymptotic functional equation for $L_n$ of the usual type $s\longleftrightarrow 1-s$. A weaker (additive) version of the latter holds unconditionally by Proposition \ref{MainProp}, but the remarkable thing is that a strengthening of this property is ruled by GRH.
\\

\noindent
\textbf{Acknowledgement. }The author gratefully thanks Anders Karlsson for useful discussions, comments and corrections to this paper. This work was also partly inspired by a question raised by François Ledrappier following a lecture about \cite{FK16}.
\end{section}
\begin{section}{Asymptotics of $L_n(s,\chi)$}
In this section we derive an asymptotic formula for $L_n(s,\chi)$ when the parameter $n\rightarrow\infty$.
\begin{proof}[Proposition \ref{MainProp}]
The proof is based on an Euler-MacLaurin formula established in \cite{Sid04}. Indeed, let $0<\Re(s)<1$ and apply Theorem 2.1 in \cite{Sid04} to the function $f(x)=\frac{1}{(\sin(\pi x))^{s}}$, with $a=0$ and $b=1$ to get (for any $0<\theta<1$) \begin{align*}\frac{1}{n}\sum_{j=0}^{n-1}\frac{1}{(\sin(\frac{\pi}{n}(j+\theta)))^{s}}= \int_0^1\frac{1}{(\sin(\pi x))^{s}}dx + \pi^{-s}\left(\zeta(s,\theta)+\zeta(s,1-\theta)\right)\frac{1}{n^{1-s}}\\+\pi^{2-s}\frac{s}{6}\left(\zeta(s-2,\theta)+\zeta(s-2,1-\theta)\right)\frac{1}{n^{3-s}}\\+O\left(\frac{1}{n^{5-s}}\right),
\end{align*}
since near $0$ we have $f(x)=x^{-s}(\pi^{-s}+\frac{s}{6}\pi^{2-s}x^2+O(x^4))$. Here $\zeta(s,\theta)$ denotes the Hurwitz zeta function.\\
Now we notice that $$L_n(s,\chi)=\sum_{r=1}^{k-1}\chi(r)\sum_{j=0}^{n-1}\frac{1}{\sin(\frac{\pi}{n}(j+\frac{r}{k}))^{s}}$$ using periodicity of characters. The following identity is well-known: 
\begin{equation}\label{L-Hurwitz}
	L(s,\chi)=\frac{1}{k^s}\sum_{m=1}^k\chi(m)\zeta(s,m/k),
\end{equation} 
with $\chi$ as before and $s\neq 1$. Putting everything together, we obtain 
\begin{align*}
L_n(s,\chi)&=\left(\frac{n}{\pi}\right)^{s}\sum_{r=1}^{k-1}\chi(r)\left(\zeta(s,r/k)+\zeta(s,1-r/k)\right)\\&+\left(\frac{n}{\pi}\right)^{s-2}\frac{s}{6}\sum_{r=1}^{k-1}\chi(r)\left(\zeta(s-2,r/k)+\zeta(s-2,1-r/k)\right)+O\left(n^{s-4}\right)\\ &=2\left(\frac{n}{\pi}\right)^{s}\sum_{r=1}^{k-1}\chi(r)\zeta(s,r/k)+2\left(\frac{n}{\pi}\right)^{s-2}\frac{s}{6}\sum_{r=1}^{k-1}\chi(r)\zeta(s-2,r/k)+O\left(n^{s-4}\right),
\end{align*}
where we used the fact that $$\sum_{r=1}^{k-1}\chi(r)\zeta(s,1-\frac{r}{k})=\sum_{r=1}^{k-1}\chi(k-r)\zeta(s,\frac{k-r}{k})=\sum_{r=1}^{k-1}\chi(r)\zeta(s,\frac{r}{k}).$$
Now we can apply (\ref{L-Hurwitz}) and conclude by analytic continuation.
\end{proof}
\begin{remark}
\emph{We can obtain as many terms as we want in the proposition by looking further in the asymptotics of $f$ near $0$.}
\end{remark}
\end{section}
\begin{section}{Relation with GRH}
A natural question to ask is whether there is a kind of functional equation for $L_n$. Indeed, Proposition \ref{MainProp} trivially implies $\lim_{n\rightarrow\infty}\left(\xi_n(s,\chi)-\xi_n(1-s,\overline{\chi})\right)=0$. But \cite{FK16} suggests looking for a stronger version and an equivalence to GRH. This is the content of Theorem \ref{MainThm} which we now begin to prove. 
\\

\noindent
In what follows, we write $G(\chi)=\sum_{l=0}^{k-1}\chi(l)e^{\frac{2\pi i l}{k}}$ for the Gauss sum of $\chi$. Recall that $L(s,\chi)$ satisfies a functional equation: write $\xi(s,\chi)=(\frac{\pi}{k})^{\frac{-s}{2}}\Gamma(\frac{s}{2})L(s,\chi)$ for the completed $L$-function. Then (see \cite{Cha70}), $$\xi(s,\chi)=\frac{G(\chi)}{\sqrt{k}}\xi(1-s,\overline{\chi}).$$ We begin the proof with the following lemma which concerns only Dirichlet $L$-functions.
\begin{lemma}\label{mazwsa}
Let $t\in\bbR$ be fixed with $|t|\geq 8$ and write $s=\sigma+it$. Then the function $$\Bigg|\frac{L(s+2,\chi)}{L(s-2,\chi)}\Bigg|$$ is strictly increasing in $0<\sigma<1$.
\end{lemma}
\begin{proof}
The proof relies on a recent paper by Matiyasevich, Saidak and Zvengrowski (\cite{MSZ14}). By Lemma 2.1 in that paper, it is enough to prove that the real part of the logarithmic derivative of $\frac{L(s+2,\chi)}{L(s-2,\chi)}$ is positive for all $s$ as in the statement of the lemma. Thus we are done if we show that $$\Re\left(-\frac{L'(s-2,\chi)}{L(s-2,\chi)}\right)>\Bigg|\Re\left(\frac{L'(s+2,\chi)}{L(s+2,\chi)}\right)\Bigg|.$$
The right-hand side is easy to estimate. Indeed, since $\Re(s+2)>1$ we have \begin{align*}\Bigg|\Re\left(\frac{L'(s+2,\chi)}{L(s+2,\chi)}\right)\Bigg|&=\Bigg|-\Re\left(\sum_{n\geq 1}\frac{\Lambda(n)\chi(n)}{n^{s+2}}\right)\Bigg|\leq\Bigg|\sum_{n\geq 1}\frac{\Lambda(n)\chi(n)}{n^{s+2}}\Bigg| \\ &\leq \sum_{n\geq 1}\frac{\Lambda(n)}{n^{\sigma+2}}\leq \sum_{n\geq 1}\frac{\Lambda(n)}{n^2} = -\frac{\zeta'(2)}{\zeta(2)}<0.57.
\end{align*}
For the left-hand side we first use the definition of $\xi$ to see that $$\Re\left(\frac{\xi'(s-2,\chi)}{\xi(s-2,\chi)}\right)=\frac{1}{2}\log \left(\frac{k}{\pi}\right)+\frac{1}{2}\Re\left(\psi\left(\frac{s-2}{2}\right)\right)+\Re\left(\frac{L'(s-2,\chi)}{L(s-2,\chi)}\right),$$ where $\psi(z):=\frac{\Gamma'(z)}{\Gamma(z)}$ is the logarithmic derivative of $\Gamma$.
\\
Since $0<\sigma<1$ we have $\Re\left(\frac{\xi'(s-2,\chi)}{\xi(s-2,\chi)}\right)<0$, as is shown in \cite{MSZ14} using the Hadamard product. Thus, \begin{align*}
\Re\left(-\frac{L'(s-2,\chi)}{L(s-2,\chi)}\right)&>\frac{1}{2}\log \left(\frac{k}{\pi}\right)+\frac{1}{2}\Re\left(\psi\left(\frac{s-2}{2}\right)\right)\\&>\frac{1}{2}\Re\left(\psi\left(\frac{s-2}{2}\right)\right)-0.03\\&>0.57,
\end{align*}
where we used $k\geq 3$ in the second inequality and Lemma 3.2 (v) of \cite{MSZ14} (with our hypothesis that $|t|\geq 8$) in the last one.
\end{proof}
We now proceed to prove the equivalence stated in Theorem \ref{MainThm}.
\begin{proof}[Theorem \ref{MainThm}]
By Proposition \ref{MainProp} we have
\begin{align}\label{asympt}
\xi_n(s,\chi)&=2\left(\frac{\pi}{k}\right)^{-\frac{s}{2}}\Gamma\left(\frac{s}{2}\right)\Bigg(L(s,\chi)+\frac{s}{6}\left(\frac{kn}{\pi}\right)^{-2}L(s-2,\chi)+O\left(\frac{1}{n^4}\right)\Bigg)\nonumber\\&= 2\xi(s,\chi)+\alpha(s,\chi)\frac{1}{n^2}+O\left(\frac{1}{n^4}\right),
\end{align} 
where we denote $\alpha(s,\chi):=\frac{s}{3}(\frac{\pi}{k})^{2-\frac{s}{2}}\Gamma(\frac{s}{2})L(s-2,\chi)$.
\\
Since $G(\chi)$ has modulus $\sqrt{k}$, the functional equation for $\xi$ together with (\ref{asympt}) implies that part (i) of the theorem is true for any $s$ which is not a zero of $L$. We want to know when it is true also for other values of $s$, namely those corresponding to zeros of $L$. For this we observe that 
\begin{align*}
\Bigg|\alpha(1-s,\overline{\chi})\Bigg|&=\Bigg|\frac{1-s}{3}\left(\frac{\pi}{k}\right)^{2-\frac{1-s}{2}}\Gamma\left(\frac{1-s}{2}\right)L(-s-1,\overline{\chi})\Bigg|\\&=\Bigg|\frac{1-s}{3}\left(\frac{\pi}{k}\right)^{\frac{3}{2}+\frac{s}{2}}\Gamma\left(\frac{1-s}{2}\right)\xi(-s-1,\overline{\chi})\left(\frac{\pi}{k}\right)^{\frac{-s-1}{2}}\frac{1}{\Gamma(\frac{-s-1}{2})}\Bigg|\\&=\Bigg|\frac{(s-1)(s+1)}{6}\frac{\pi}{k}\Gamma\left(\frac{s+2}{2}\right)L(s+2,\chi)\left(\frac{\pi}{k}\right)^{-\frac{s+2}{2}}\Bigg|\\&=\Bigg|\frac{s(s-1)(s+1)}{12}\left(\frac{\pi}{k}\right)^{-\frac{s}{2}}\Gamma\left(\frac{s}{2}\right)L(s+2,\chi)\Bigg|
\end{align*}
and so the identity \begin{equation}\label{alpha}
|\alpha(s,\chi)|=|\alpha(1-s,\overline{\chi})|
\end{equation} is equivalent to the following:
\begin{equation}\label{alphaequiv}
\Bigg|\frac{L(s+2,\chi)}{L(s-2,\chi)}\Bigg|=\frac{4\pi^2}{k^2|s^2-1|}.
\end{equation}
We note that identities (\ref{alpha}) and (\ref{alphaequiv}) are true when $\Re(s)=\frac{1}{2}$, because $$\alpha(1-(1/2+it),\overline{\chi})=\alpha(\overline{1/2+it},\overline{\chi})=\overline{\alpha(1/2+it,\chi)}.$$
Clearly, the right-hand side of (\ref{alphaequiv}) is strictly decreasing in $0<\sigma<1$. By Lemma \ref{mazwsa} the left-hand side is strictly increasing in $0<\sigma<1$ provided $\Im(s)\geq 8$. Thus, (\ref{alphaequiv}) can be true for only one value of $\sigma$ and this is $\sigma=\frac{1}{2}$.
\\

\noindent
This finishes the proof of Theorem \ref{MainThm}, since  when $s$ is a zero of $L$, the assertion $$\lim\limits_{n \to \infty}\frac{|\xi_n(s,\chi)|}{|\xi_n(1-s,\overline{\chi})|}=1$$ is equivalent to (\ref{alpha}).
\end{proof}
\end{section}
\begin{section}{On real zeros of $L(s,\chi)$}
It is an important and difficult problem to show that Dirichlet $L$-functions have no real zeros. Even partial answers to this question have many consequences, in particular in the theory of primes in arithmetic progressions and in the theory of quadratic forms and class numbers initiated by Gauss. Particular cases such as the non-vanishing of $L(\frac{1}{2},\chi)$ or the existence of a Siegel zero are much studied and still open in general.
\\

\noindent
Consider $\chi$ a primitive, even and \emph{real} character modulo $k\geq 3$ (if $k$ is odd this is the Jacobi symbol modulo $k$). It is interesting to observe that if we suppose that $0<s<1$ is a real zero of the associated Dirichlet $L$-function then Proposition \ref{MainProp} tells us that $L_n(s,\chi)<0$ for all $n$ sufficiently large, since $L(s,\chi)$ is negative for $-2<s<0$. Thus the possible existence of real zeros of $L(s,\chi)$ is encoded in the $L$-function of a very simple graph.
\begin{corollary}\label{positivity}
Let $\chi$ be a primitive, even and real character modulo $k\geq 3$ and let $0<s<1$. The following are equivalent:
\item[(i)] $L(s,\chi)>0$;
\item[(ii)] For infinitely many $n\geq 1$ we have $L_n(s,\chi)\geq 0$, that is $\sum_{j=1}^{kn-1}\frac{\chi(j)}{\sin(\frac{\pi j}{kn})^{s}}\geq 0$.
\end{corollary}
A similar result is avalaible if we consider the standard partial sums $\sum_{j=1}^{kn-1}\frac{\chi(j)}{j^s}$ instead, see \cite{Chu05}. We refer to \cite{FK16} for a heuristic explanation as to why these partial sums with sines, although being seemingly more involved, may be an interesting alternative to the usual partial sums in the study of Dirichlet $L$-functions.
\\

\noindent
Now we would like to mention an interesting problem related to the study of the sign of $L_n$ and say a few words about its analog in the classical case. The problem comes from the following simple observation: if we write $$\frac{1}{(1-x)^s}=\sum_{m\geq 0}a_m(s)x^m$$ for $0<s<1$, we see that $$\sum_{j=1}^{kn-1}\frac{\chi(j)}{\sin(\frac{\pi j}{kn})^{s}}=\sum_{m\geq 1}a_m\left(s/2\right)\sum_{j=1}^{kn-1}\chi(j)\cos^{2m}\left(\frac{\pi j}{kn}\right).$$
It is then natural to ask:
\begin{question}\label{q1}
Let $\chi$ be a primitive, even and real character modulo $k\geq 3$. Does there exist a sequence $(a_n)$ of positive integers such that \begin{equation}\label{posofcos}
\sum_{j=1}^{ka_n-1}\chi(j)\cos^{2m}\left(\frac{\pi j}{ka_n}\right)\geq 0
\end{equation}
for every $m\geq 1$?
\end{question}
Note that this claim is independent of $s$. Since $a_m(s/2)>0$ for all $0<s<1$ and $m\geq 1$, a positive answer to Question \ref{q1} would imply that $L(s,\chi)$ has no real zero. Obviously it is a much stronger statement but it is not clear to us whether it is true or not and we think it is an interesting problem to investigate.
\\

\noindent 
It may be a good idea to draw a comparison with the classical case, that is when we consider the sign of the partial sums $\sum_{j=1}^{kn-1}\frac{\chi(j)}{j^s}$. We already mentioned the paper \cite{Chu05} where it is shown that it is enough to establish that $\sum_{j=1}^{k-1}\frac{\chi(j)}{j^s}$ is positive in order to prove that there is no real zero. We can make the same computation as above: $$\sum_{j=1}^{k-1}\frac{\chi(j)}{(\frac{j}{k})^s}=\sum_{m\geq 0}a_m(s)\sum_{j=1}^{k-1}\chi(j)\left(1-\frac{j}{k}\right)^m=\sum_{m\geq 0}a_m(s)\sum_{j=1}^{k-1}\chi(j)\left(\frac{j}{k}\right)^m,$$ where we used the fact that $\chi$ is even in the last equality. Thus if $$S(m,\chi):=\sum_{j=1}^{k-1}\chi(j)j^m$$ was non-negative for all $m\geq 0$ (it is well-known that $S(0,\chi)=S(1,\chi)=0$) it would imply that $L(s,\chi)$ has no real zero. This was already observed by Rosser in \cite{Ros49}. The analog to Question \ref{q1} is then:
\begin{question}\label{q2}
Let $\chi$ be a primitive, even and real character modulo $k\geq 3$. Is it true that $$S(m,\chi)\geq 0$$ for every $m\geq 1$?
\end{question}
It is possible to show that $S(m,\chi)>0$ if $m\geq k-2$, simply by showing that the term with $j=k-1$ dominates the rest of the sum. Thus if we are given a character only a finite number of sums need to be checked in order to answer Question \ref{q2}. On the other hand, there is a nice formula available for $S(m,\chi)$, probably well-known, which we prove now.
\begin{proposition}
Let $\chi$ be a primitive, even and real character modulo $k\geq 3$. For any $n\geq 1$ and $m\geq 2$ we have
\begin{equation}\label{Faulhaber}
\frac{1}{(kn)^m}\sum_{j=1}^{kn-1}\chi(j)j^m=2n\sqrt{k}\sum_{j=1}^{[\frac{m}{2}]}\frac{(-1)^{j+1}m(m-1)\ldots (m-2j+2)}{(4\pi^2 n^2)^j}L(2j,\chi).
\end{equation}
\end{proposition} 
This can be thought of as a kind of Faulhaber formula twisted by a Dirichlet character.
\begin{proof}
We use once again Euler-MacLaurin formula (see for example \cite{Ste50}):
\begin{align*}
\sum_{j=1}^{kn-1}\chi(j)\left(\frac{j}{kn}\right)^m &= \sum_{r=1}^{k-1}\chi(r)\sum_{j=0}^{n-1}\left(\frac{1}{n}(j+\frac{r}{k})\right)^m\\&=\sum_{r=1}^{k-1}\chi(r)\left(\frac{n}{m+1}+\sum_{j=1}^{m}\frac{(-1)^{j+1}\zeta(1-j,\frac{r}{k})}{(j-1)!n^{j-1}}m(m-1)\ldots (m-j+2)\right)\\&=-\sum_{j=1}^{[\frac{m}{2}]}\frac{m(m-1)\ldots (m-2j+2)}{(2j-1)!(kn)^{2j-1}}L(1-2j,\chi),
\end{align*}
where in the last equality we used (\ref{L-Hurwitz}) and the fact that $L(-2m,\chi)=0$ for $m\geq 0$ when $\chi$ is even. We obtain (\ref{Faulhaber}) by applying the functional equation.
\end{proof}
\begin{corollary}\label{corFaul}
Let $\chi$ be a primitive, even and real character modulo $k\geq 3$. For $m=2,3,4,5,6,7$ we have $$S(m,\chi)>0.$$
\end{corollary}
\begin{proof}
Apply (\ref{Faulhaber}) with $n=1$ and $m=2,3,4,5,6$ or $7$ and use the bounds $$\frac{\zeta(2s)}{\zeta(s)}\leq L(s,\chi)\leq \zeta(s),$$ valid for all $s>1$ and all real characters.
\end{proof}
With a more delicate analysis it is in principle possible to extend Corollary \ref{corFaul} to greater values of $m$ but it becomes more and more difficult as $m$ grows. To our knowledge, no example of a character $\chi$ as in the corollary and an integer $m$ such that $S(m,\chi)<0$ has been found.
\\

\noindent 
The same problem for an \emph{odd} character has received much more attention, starting with a paper by Chowla, Ayoub and Walum (\cite{ACB67}) where they show that the sign of $S(3,\chi)$ changes infinitely often when $\chi$ varies. See \cite{Fin70} for a stronger result and \cite{TW99} for subsequent work on this problem. These kind of sums also appear in \cite{Ber76}. Of course, it makes no sense in our situation to consider odd characters since $L_n(s,\chi)\equiv 0$ in that case. We believe that the study of the sign of both sums $S(m,\chi)$ for even $\chi$ and (\ref{posofcos}) are worthy problems and we hope this note will stimulate work on this topic.
\\

\bibliographystyle{plain}
\bibliography{bib-L}
\noindent Fabien Friedli 

\noindent Section de mathématiques \\
Université de Genève

\noindent 2-4 Rue du Lièvre\\
Case Postale 64

\noindent 1211 Genève 4, Suisse 

\noindent e-mail: fabien.friedli@unige.ch

\end{section}
\end{document}